\documentclass{article}    
\usepackage[a4paper,left=2cm,right=2cm,top=2cm,bottom=2cm]{geometry}                
\usepackage{graphicx}
\usepackage{tikz}
\usetikzlibrary{patterns,decorations,decorations.pathmorphing}

\usepackage{amsmath}
\usepackage{hyperref}
\usepackage{authblk}
\usepackage{amssymb}
\newtheorem{theorem}{Theorem}
\newtheorem{conjecture}{Conjecture}
\newtheorem{lemma}{Lemma}
\newtheorem{claim}{Claim}
\newcommand{\keywords}[1]{\par\textbf{Keywords:} #1}
\def\Box{\raisebox{3pt}{\framebox{\hbox to 3pt{\vbox to 3pt{}}}}}
\newenvironment{proof}{\medskip \noindent{\sc Proof:}}{\quad$\Box$\par\medskip} 
\DeclareMathOperator{\opt}{opt}

\DeclareMathOperator{\ifcov}{IFCov}
\DeclareMathOperator{\Cov}{Cov}

\begin{document}

\title{Constructions for the optimal pebbling of grids}

\author[1,2]{Ervin Gy\H{o}ri\thanks{gyori.ervin@renyi.mta.hu}}
\author[3,4]{Gyula Y. Katona\thanks{kiskat@cs.bme.hu}}
\author[3]{L\'aszl\'o F. Papp\thanks{lazsa@cs.bme.hu}}
\affil[1]{Alfr\'ed R\'enyi Institute of Mathematics, Budapest, Hungary}
\affil[2]{Department of Mathematics, Central European University, Budapest, Hungary}
\affil[3]{Department of Computer Science and
Information Theory, Budapest University of Technology and Economics, Hungary}
\affil[4]{MTA-ELTE Numerical Analysis and Large Networks Research Group, Hungary}

\date{Received: date / Accepted: date}

\maketitle

\begin{abstract}
  In [\textit{Xue, Yerger: \emph{Optimal Pebbling on Grids}, Graphs and Combinatorics \textbf{(32)} no.~3}] 
the authors conjecture that if every vertex of an
  infinite square grid is reachable from a pebble distribution, then
  the covering ratio of this distribution is at most $3.25$. First we
  present such a distribution with covering ratio $3.5$, disproving
  the conjecture. The authors in the above paper 
also claim to prove that
  the covering ratio of any pebble distribution is at most
  $6.75$. The proof contains some errors. We present a
  few interesting pebble distributions that this proof does not seem
  to cover and highlight some other difficulties of this topic.

\end{abstract}
\section{Introduction}
\label{intro}
Graph pebbling has its origin in number theory. It is a model for the
transportation of resources. Starting with a pebble distribution on
the vertices of a simple connected graph, a \emph{pebbling move}
removes two pebbles from a vertex and adds one pebble at an adjacent
vertex. We can think of the pebbles as fuel containers. Then the loss
of the pebble during a move is the cost of transportation. A vertex is
called \emph{reachable} if a pebble can be moved to that vertex using
pebbling moves. There are several questions we can ask about
pebbling. One of them is: How can we place the smallest number of
pebbles such that every vertex is reachable? The minimum number of pebbles in 
such a pebble distribution is called the \emph{optimal pebbling
  number} of the graph. The \emph{optimal covering ratio of a graph} is the number of vertices of the graph  divided by the optimal 
	pebbling number. Moreover, the \emph{covering ratio of an arbitrary distribution} is the number of vertices reachable from the distribution divided by the number of pebbles in the distribution.
	For a 
	comprehensive list of references for the extensive
literature see the survey papers
\cite{Hurlbert_survey1,Hurlbert_survey2,Hurlbert_survey3}.

In Section~\ref{sec:1} we show a pebble distribution which disproves a conjecture of Xue and Yerger. Section~\ref{sec2} and \ref{sec3} contain some interesting counterexamples for some lemmas stated in \cite{XY}. We also mention some phenomenons why we think that the proof of Theorem 8 of \cite{XY}  can not be corrected, and a proper proof requires a different approach. 

In the last section we introduce a new problem, called \emph{optimal integer fractional covering ratio}, where the tools introduced in \cite{XY} can be used. We give a lower and an upper bound on the optimal integer fractional covering ratio of large grids.

\section{Definitions}
In this section we summarize the definitions which the paper uses. We start with basic ones, which are well known in the area of pebbling, then continue with more complicated ones, which were introduced in \cite{XY}.

\subsection{Traditional pebbling}

A \emph {pebbling distribution} $D$ is a $V(G)\rightarrow \mathbb{N}$ function. If $D(v)\geq 2$ and $u$ and $v$ are adjacent vertices, then we can apply a $(v\rightarrow u)$ \emph{pebbling move}. It decreases $D(v)$ by two and increases $D(u)$ by one. A vertex $v$ is \emph{reachable} under $D$ if either $D(v)\geq 1$ or we can apply a sequence of pebbling moves such that the last one is an $(u\rightarrow v)$ move.
 
A distribution is \emph{solvable} if each vertex is reachable under it. We use $|D|$ for the \emph{size of distribution} $D$, which is the total number of pebbles placed on the graph.
A distribution is \emph{optimal} on graph $G$ if its size is minimal among all solvable distributions of $G$. \emph{The optimal pebbling number} is the size of an optimal distribution and it is denoted by $\pi_{\opt}(G)$.

The \emph{coverage} of distribution $D$ is the set of reachable vertices. We denote the size of this set by $\Cov(D)$. 

The \emph{covering ratio} of $D$ is defined as $\frac{\Cov(D)}{|D|}$.   

\subsection{Infinite graphs}
The infinite square grid is denoted by $G_{inf}$.
In \cite{XY} the authors talk about $G_{inf}$, they do not provide a proper definition for the covering ratio in the case when the distribution is infinite. In their reasoning they assume that the number of pebbles is finite, therefore we will assume the same. When solvability comes into play, we think about arbitrary large, but finite square grids, whose border's size is marginal compared to their total size. Therefore their covering ratio is well defined. We define the optimal covering ratio of $G_{inf}$ as the limit of larger and larger square grids's optimal covering ratios.

\begin{claim}
This limit exists.
\end{claim}

\begin{proof}
We prove, that the reciprocal of the series is convergent. We denote the $n\times n$ grid by $G_{n\times n}$.

Let $\epsilon$ be an arbitrary positive real number and $A=\liminf_{n\rightarrow \infty} \frac{\pi_{\opt}(G_{n\times n})}{n^2}$. Choose $m$ large enough such that $\frac{6}{m}<\epsilon$, $\frac{\pi_{\opt}(G_{m\times m})}{m^2}-A\leq \frac{\epsilon}{2}$ and for every $n\geq m$ the inequality $A-\frac{\pi_{\opt}(G_{n\times n})}{n^2}<\epsilon$ holds. We show that $n\geq m^2$ implies $\left|A-\frac{\pi_{\opt}(G_{n\times n})}{n^2}\right|<\epsilon$.

Write $n$ as $km+r$ and partition $G_{n\times n}$ to $k^2$ piece of disjoint $G_{m\times m}$ and $r^2+2rkm$ remaining vertices. We use its optimal distribution on each $G_{m\times m}$ and place one pebble to each remaining vertex. In such a way we obtain a solvable distribution of $G_{n\times n}$. Therefore: 

$$-\epsilon<\frac{\pi_{\opt}(G_{n\times n})}{n^2}-A\leq \frac{k^2\pi_{\opt}(G_{m\times m})+r^2+2rkm}{k^2m^2+r^2+2rkm}-A\leq \frac{\pi_{\opt}(G_{m\times m})}{m^2}-A+\frac{3}{m}<\frac{\epsilon}{2}+\frac{\epsilon}{2}=\epsilon$$

 \end{proof}

We note, that there are several ways to define pebbling parameters for  infinite graphs by considering infinite distributions. Nevertheless, it is beyond the scope of this paper. 

\subsection{Combining distributions}

Assume that we have two distributions $D$ and $D'$. We say that these two distributions \emph{interact} at vertex $v$ if $v$ is reachable under both. Vertex $v$ is a \emph{boundary vertex} of $D$ if $v$ is reachable under $D$ but one of its neighbours is not.

It is a natural idea to unify two distributions $D$ and $D^*$ by placing $D(v)+D^*(v)$ pebbles at $v$, to create a bigger one. $D$ and $D^*$ are stronger together in the sense, that some vertices are not reachable under $D$ nor $D^*$, but they are reachable under the $D'$ which we get by combining them. This phenomenon requires the presence of interacting vertices. 

For example, if an interaction vertex is boundary in both $D$ and $D^*$ and one of its neighbours is not reachable under $D$ and $D^*$, then this neighbour is reachable under $D'$.

A \emph{unit} is a vertex having at least one pebble. A \emph{unit distribution} contains only one unit.
 
Using the combination method we can build any distribution from unit distributions. It often happens, that the coverage of a unit is disjoint from coverage of the rest of the distribution. So the unit and the rest do not share an interaction vertex and they can be handled separately. We say that these units are \emph{lonely}.

We can also ask that what is the difference between the covering ratios of $D$ and $D'$. Of course, it depends on $D^*$, but we would like to measure it. This motivates the definition of \emph{marginal covering ratio}, which is the following:


$$\frac{\Cov(D')-\Cov(D)}{|D'|-|D|}.$$

We can not compute the covering ratio of $D'$ if we know the covering ratio of $D$ and the marginal covering ratio. On the other hand, we can state upper bounds, which we are interested in.

\subsection{Fractional pebbling}
A variation of the pebbling problem if we allow fractional pebbles. This leads to the area of fractional pebbling, which is well studied in \cite{frac}. 

A \emph{continuous distribution} on $G$ is a $V(G)\rightarrow \mathbb{R}^+\cup\{0\}$ function. A \emph{continuous pebbling move} removes $t$ pebbles from a vertex and place  $t/2$ pebbles  to an adjacent vertex, where $t$ can be any positive real number.

The \emph{optimal fractional pebbling} number is the size of the smallest solvable continuous distribution. It can be calculated by solving a linear program and it is a lower bound on the optimal pebbling number.


Let $D$ be a continuous distribution. The \emph{weight function} of $D$, which is defined on the vertex set of $G$, is defined as: 
$$W_D(u)=\sum_{v\in V(G)}D(v)2^{-d(u,v)}.$$
It tell us that how many pebbles can be moved to a vertex under $D$ by continuous pebbling moves. It is a very useful tool, widely used by several authors in this topic.

In a solvable distribution the weight of each vertex is at least one. However, if we are considering optimal pebbling distributions, then several vertices have more weight than one. We can consider this extra weight as an excess, and try to calculate it. The sum of these values gives an estimate on the difference between the fractional and the traditional optimal pebbling numbers.    

So we define the \emph{excess weight} function as:
$$\widehat{W}_D(u)=\begin{cases}W_D(u)-1 &\text{if } W_D(u)>1,\\
W_D(u) &\text{if } W_D(u)\leq 1.
\end{cases}$$

With the help of this function, we can give a fractional generalization of covering ratio. To construct the numerator we count the number of reachable vertices and also add the weight of not reachable vertices. 
The \emph{covering ratio ceiling} of $D$ is:

$$\frac{\sum_{v\in V(G)}W_D(v)-\sum_{v\in V(G)}\widehat{W}_D(v)}{\sum_{v\in V(G)}D(v)}.$$

Like the marginal covering ratio we also define a quantity which measures in some way the change of covering ratio ceiling in case of adding some extra pebbles. So let $D$ and $D'$ be distributions, such that $D(v)\leq D'(v)$ for each vertex. The \emph{marginal covering ratio ceiling} of these distributions is defined by

$$\frac{\left(\sum_{v\in V(G)}W_D'(v)-\sum_{v\in V(G)}\widehat{W}_D'(v)\right)-\left(\sum_{v\in V(G)}W_D(v)-\sum_{v\in V(G)}\widehat{W}_D(v)\right)}       {\sum_{v\in V(G)}D'(v)-\sum_{v\in V(G)}D(v)}.$$


\section{Lower bound for the optimal covering ratio of the grid}
\label{sec:1}

Conjecture 2 in \cite{XY} states that if every vertex of an infinite
square grid is reachable from a pebble distribution, then the covering
ratio of this distribution is at most $3.25$. 

We present a sequence of distributions on big grids whose covering ratios converge to $3.5$, disproving the conjecture. 
Repeating periodically the optimal distribution of $G_{n\times n}$ results a solvable distribution on the infinite grid. Therefore considering real infinite graphs and distributions can not decrease the optimal covering ratio, if it is defined intuitively.

A
distribution is shown on Figure~\ref{fig:27distr},  units consisting of four pebbles
are placed to every other vertex on every 7th diagonal. It is easy to
calculate that the covering ratios of such distributions tends to
$7/2=3.5$. The arrows and the shaded areas on the figure indicate how
to reach all vertices of the grid from this distribution. We
conjecture that this is best possible.

  \begin{figure}
    \centering
    \begin{tikzpicture}[scale=.45]
\input{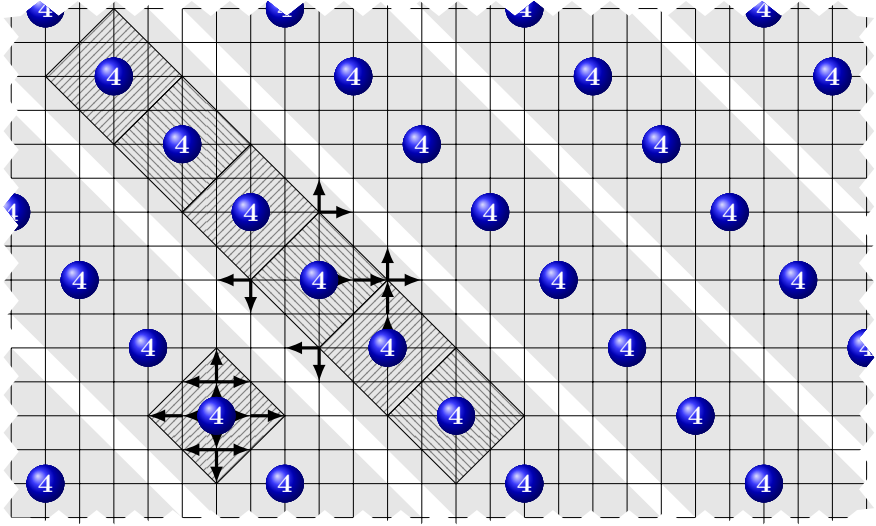}

 \clip[decorate, decoration={zigzag,segment length=5mm,amplitude=1mm}](-1,-1) rectangle (24,14);
\foreach \x in {0,...,25} 
\foreach \y in {0,...,14}
\coordinate (X\x_Y\y) at (\x,\y);

\foreach \i in {-18,-11,-4,3,10,17,24}
\fill[color=black!10]  (\i,15) -- +(6,0) -- +(22,-16) -- +(16,-16);
\draw[pattern=north east lines, pattern color=black!50] (7,2)  -- +(-2,2) -- +(-4,0) -- +(-2,-2) --  (7,2);

\foreach \i in {1,...,3}{
\draw[pattern=north west lines, pattern color=black!50] (4*\i,16-4*\i)  -- +(2,-2) -- +(0,-4) -- +(-2,-2)  --  (4*\i,16-4*\i);
\draw[pattern=north east lines, pattern color=black!50] (4*\i,16-4*\i)  -- +(-2,2) -- +(-4,0) -- +(-2,-2) --  (4*\i,16-4*\i);
}
\draw [very thin, step=1cm] (-1.5,-1.5) grid (25.5,14.5);


\foreach \s/\t in {
X5_Y2/X5_Y3,
X5_Y2/X6_Y2,
X5_Y2/X4_Y2,
X5_Y2/X5_Y1,
X5_Y3/X5_Y4,
X5_Y3/X6_Y3,
X5_Y3/X4_Y3,
X5_Y1/X5_Y0,
X5_Y1/X6_Y1,
X5_Y1/X4_Y1,
X4_Y2/X3_Y2,
X6_Y2/X7_Y2,
X8_Y8/X8_Y9,
X8_Y8/X9_Y8,
X10_Y6/X11_Y6,
X10_Y6/X10_Y7,
X6_Y6/X6_Y5,
X6_Y6/X5_Y6,
X8_Y4/X8_Y3,
X8_Y4/X7_Y4,
X10_Y4/X10_Y5,
X10_Y5/X10_Y6,
X8_Y6/X9_Y6,
X9_Y6/X10_Y6}
\draw [nyil] (\s) -- (\t);
\foreach \x in {-1,...,7} 
	\foreach \y in  {-3,...,3}  {
	\pgfmathtruncatemacro{\i}{(2*\x+7*\y)}
	\pgfmathtruncatemacro{\j}{(14-2*\x)}
\ifnum \i>-2
\ifnum \i<26 
\node [pebble] (N_\x_\y) at (\i,\j) {$\mathbf{4}$}
\fi
\fi;

}

\foreach \p in {0,...,5} {
\pgfmathtruncatemacro{\q}{(\p+1)}
}
\end{tikzpicture}
    \caption{Pebble distribution of the grid with covering ratio $3.5$.}
    \label{fig:27distr}
  \end{figure}

\begin{conjecture}
If every vertex of an infinite
square grid is reachable from a pebble distribution, then the covering
ratio of this distribution is at most $3.5$.
\end{conjecture}

\section{Comments on the upper bound for the optimal covering ratio
  of the grid}\label{sec2}

In \cite[Section 6]{XY} the authors claim to prove that the optimal
covering ratio of the grid is at most $6.75$ (see \cite[Theorem
8]{XY}).  The proof contains some errors. Although some
of these errors may be corrected somehow, in our opinion some others
cannot be corrected. We do not see how to complete the proof, but we
believe that the statement is true. In forthcoming paper \cite{GKPT}
we are to prove a better bound: the covering ratio is at most
$6.5$. In the rest of this note we point out some errors in the proof
and show some interesting pebble distributions that highlight the
difficulties of this problem.

First we summarize the above mentioned proof. It is an inductive proof
on the number of units contained in the distribution. First, as the
base case, it is shown that the theorem holds for one unit. Now assume
that it also holds for any distribution containing $n$ units. Consider
a distribution of $n+1$ units, remove an arbitrary unit, apply the
inductive hypothesis for the remaining units. Depending on the
position of the removed unit and the remaining distribution apply
\cite[Lemma 12]{XY} or \cite[Lemma 14]{XY} to complete the
proof. (In the original paper in Section 6 all
reference to the lemmas are shifted by one, so any reference to Lemma
$i$ should be to Lemma $i-1$.)

\subsection{Comments on \cite[Lemma 12]{XY}}
\label{sec:2.1}

The lemma states: \textit{For initial distribution $D$ and unit $U$,
  if only the boundary vertices of $U$ are reachable via pebbles from
  $D$, then the marginal covering ratio of a unit on $G_{inf}$ is at
  most $4.25$.}

We present several counterexamples to this lemma. 
\begin{figure}
    \centering
    {\begin{tikzpicture}[scale=.8]
\input{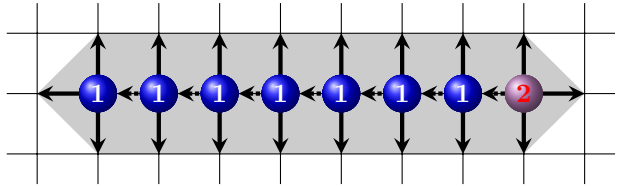}
\fill [color=black!20] (-1,2) node (v1) {} -- (0,3) node (v2) {} -- (7,3) -- (8,2) -- (7,1) -- (0,1) -- (v1);
\draw [thin, step=1cm] (-1.5,.5) grid (8.5,3.5);
\draw[->,ultra thick] (0,2) -- (-1,2);
\draw[->,ultra thick] (7,2) -- (7,3);
\draw[->,ultra thick] (7,2) -- (7,1);
\draw[->,ultra thick] (7,2) -- (8,2);
\foreach \p in {0,...,6} {
\draw[->,ultra thick] (\p,2) -- (\p,3);
\draw[->,ultra thick] (\p,2) -- (\p,1);
\node [pebble] (N_\p) at (\p,2) {$\mathbf{1}$};}
\node [pebble2] (plus) at (7,2) {$\mathbf{2}$};
\foreach \p in {0,...,5} {
\pgfmathtruncatemacro{\q}{(\p+1)}
\draw [<-, densely dotted, ultra thick]   (N_\p) -- (N_\q);
}
\draw [->,  densely dotted, ultra thick]   (plus) -- (N_6);
 
\end{tikzpicture}}
    \caption{Distribution with large marginal covering ratio.}
    \label{fig:ell1}
  \end{figure}
  In the first example (see Fig. \ref{fig:ell1}) let $D$ be the
  distribution consisting units of size 1 in a horizontal row. The
  covering ratio of $D$ is clearly $1$. Now let $U$ be a unit of size
  $2$ placed at the end of this row. Only a boundary vertex of $U$ is
  reachable from $D$ (i.e. a unit of $D$ is on the boundary of $U$),
  so the conditions of the lemma are satisfied. However after adding
  $U$, all vertices in the shaded area become reachable. Since the size
  of $U$ is constant (it is $2$), the marginal covering ratio depends
  on the size of $D$. It can be arbitrary large even if we consider
  finite $D$ distributions, and it can be infinite if $D$ is made
  infinitely large.

\begin{figure}
    \centering
    {\begin{tikzpicture}[scale=.8]
\input{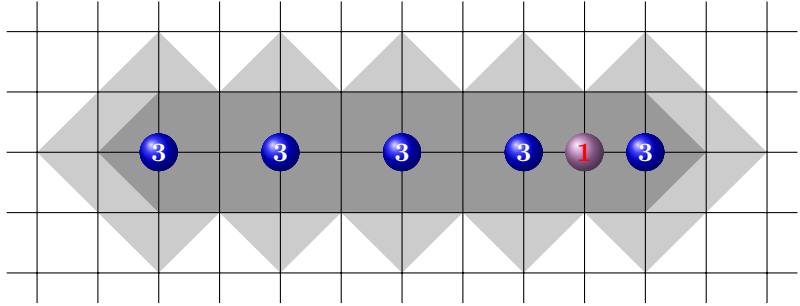}\
\foreach \r in {-2,0,2,4,6}
\fill [color=black!20] (\r,2) node (R_\r){} -- +(2,2) -- +(4,0) -- +(2,-2)-- (R_\r);

\fill [color=black!40] (-1,2) node (v1) {} -- (0,3) node (v2) {} -- (8,3) -- (9,2) -- (8,1) -- (0,1) -- (v1);
\draw [thin, step=1cm] (-2.5,-.5) grid (10.5,4.5);

\foreach \p in {0,2,4,6,8} {
\node [pebble] (N_\p) at (\p,2) {$\mathbf{3}$};
}
\node [pebble2] (plus) at (7,2) {$\mathbf{1}$};

\foreach \p in {0,...,5} {
\pgfmathtruncatemacro{\q}{(\p+1)}
}
 
\end{tikzpicture}}
    \caption{Distribution with large marginal covering ratio.}
    \label{fig:ell2}
  \end{figure}

  Another example is shown on Fig.  \ref{fig:ell2}, where a similar
  problem appears. Adding a unit of size $1$ can increase the number
  of newly reachable vertices by an arbitrary large amount.

The proof presented in \cite{XY} works only for the following weaker statement:
\begin{lemma}\label{newclaim}
  Let $D$ be a distribution and $U$ be a unit such that only the
  boundary vertices of $D$ and $U$ interact. Assume, moreover, that if
  $D$ has a unit of size one, then $U$ also has size one. Then, the
  marginal covering ratio of $U$ on $G_{{inf}}$ is at most $4.25$.
\end{lemma}

\subsection{Comments on the inductive step}

In the inductive step, we assume that there is a distribution $D$ with
covering ratio at most $6.75$. Now we add a unit $U$. The
authors do not explain this step in detail, but we assume that their
intention is to apply \cite[Lemma 12]{XY} if \textit{``only the
  boundary vertices of $U$ are reachable via pebbles from $D$''}, and
apply \cite[Lemma 14]{XY} if \textit{``the unit interacts not only on
  the boundary vertices''}.

The examples in the previous subsection show that in some cases, when
some units of size 1 are involved, \cite[Lemma 12]{XY} cannot be
used. We think that the authors intended to handle this problem with
the following sentence in the proof of \cite[Lemma 12]{XY}:
\textit{``We assume that $D$ does not contain lonely units with one
  pebble because if $D$ contains those units, we can remove them first
  and add them after $U$ has been added.''} This suggests that in the
inductive step one should be more careful how to select the unit to be
removed, remove the lonely units first. The above
example suggests that this does not work. On the other hand, it looks
promising to always remove a unit which is on the ``boundary'' of $D$,
but now the boundary is understood differently, something like the
``convex hull''. However, this approach does not look easy.

Now let us consider the case when \cite[Lemma 14]{XY} is applied. (The
lemma in fact states, that the marginal covering ratio ceiling in this
case is at most $6$, however, the proof gives $6.75$, so this is
clearly a typo). We give an example, when this fails to prove the
inductive step: Two units of size $2$ on adjacent vertices of the
grid. In this case one of the units is $D$ the other is $U$, the
interaction happens not only on the boundaries. The covering ratio of
$D$ is $2.5<6.75$, so the inductive hypothesis holds. Now \cite[Lemma
14]{XY} implies that the marginal covering ratio ceiling of $U$ is at
most $6.75$. These facts do not imply that the covering
ratio of $D\cup U$ is at most $6.75$. (Of course, the covering ratio
is in fact $2<6.75$, just the proof does not imply this.)

The covering ratio is $\Cov(D)/|D|$. However, in the definition of the
marginal covering ratio ceiling neither $\Cov(D)$ nor, more
importantly, $\Cov(D\cup U)$ appear, so it seems impossible that
these two inequalities would imply anything useful for $\Cov(D\cup U)$.
We suspect that the intention of the authors was to say that if the
marginal covering ratio \textit{ceiling} of $D$ is at most $6.75$ and
the marginal covering ratio ceiling of the pair $(D,U)$ is at most
$6.75$, then the covering ratio \textit{ceiling} of $D\cup U$
is at most $6.75$. This implication is correct, but then we do not see
how to obtain the desired bound for the covering ratio. Of course,
since the covering ratio ceiling is an upper bound for the covering
ratio, it would be enough to prove with the previous argument that the
covering ratio ceiling is at most $6.75$. The above
argument does not give this. The covering ratio ceiling of $D\cup U$
is in fact $7.25$, implying only that the covering ratio is at most
$7.25$. The basic problem is that the covering ratio ceiling of the
only unit of size 2 in $D$ is $8.5$.

\section{Connection between the covering ratio and the covering ratio ceiling}
\label{sec3}

The covering ratio ceiling is an upper bound for the covering
ratio. It is also clear that the marginal covering ratio ceiling cannot
be large if we add a new unit to a distribution. Does this imply
something about the change in the covering ratio? 

\begin{theorem} The distribution given
in Fig.~\ref{fig:ell3} shows that the covering ratio can
increase by more than one while the covering  ratio ceiling decreases.
\end{theorem}

\begin{figure}
    \centering
    {\begin{tikzpicture}[scale=.45]

\input{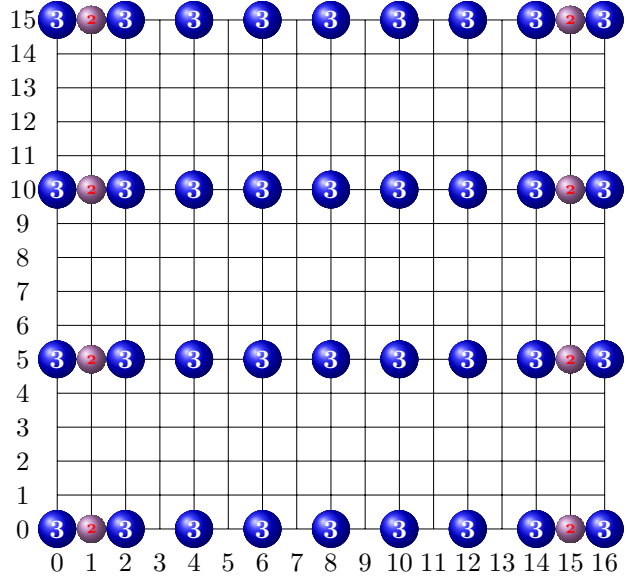}
\tikzset{koord/.style={draw}};
 \foreach \x in {0,...,15} \node  (\x) at (-1,\x) {\x};
 \foreach \x in {0,...,16} \node (\x) at (\x,-1) {\x};
\foreach \x in {0,...,16} 
\foreach \y in {0,...,15}
\coordinate (X\x_Y\y) at (\x,\y);


\draw [very thin, step=1cm] (-0,-0) grid (16,15);



\foreach \y in {0,...,3} 
	\foreach \x in  {0,...,8}	{
	\pgfmathtruncatemacro{\i}{(2*\x)}
	\pgfmathtruncatemacro{\j}{(5*\y)}
\node [pebble] (N_\x_\y) at (\i,\j) {$\mathbf{3}$};
\node [pebble2] (N_\x_\y) at (15,\j) {\tiny$\mathbf{2}$};
\node [pebble2] (N_\x_\y) at (1,\j) {\tiny$\mathbf{2}$};
}

\foreach \p in {0,...,5} {
\pgfmathtruncatemacro{\q}{(\p+1)}
}
\end{tikzpicture}}
    \caption{Pebble distribution with increasing covering ratio.}
    \label{fig:ell3}
  \end{figure}

\begin{proof}	
  Let us consider the distribution in Figure \ref{fig:ell3}. A unit of size $3$ on every second
  vertex in a row of length $2n+1$ ($n+1$ such units in a row), repeated in
  every fifth row (see
  Fig.~\ref{fig:ell3}) in $5m+1$ rows. The marginal covering ratio ceiling of a unit
   is $\frac{\Delta W-\Delta\widehat{W}}{\Delta |D|}$, where
  $\Delta W$ is the sum of the changes of the weights which all pebbles
  contributes to a vertex, $\Delta\widehat{W}$ is the same for the
  weight ceiling function, and $\Delta |D|$ is the number of added
  pebbles. 
	
  Let us first see the weights at all vertices in this
  distribution. For example the total weight for every vertex in row
  $0$ is clearly at least $3$.  Thus the total weight at every vertex in
  row $1$ is at least $3/2$. The weight for every vertex in row 2 is
  at least $3/4+3/8=9/8$, since the pebbles in row $0$ contribute at least $3/4$
  and pebbles in row $5$ contribute at least $3/8$ to its weight.  The same
  is true for all other rows, therefore it is clear that the weight of
  every vertex is at least $1$.

  If the weight of a vertex is at least $1$ then any positive change in the
  total weight will result in the same amount of positive change of
  $\widehat{W}$. So these vertices contribute $0$ to $\Delta
  W-\Delta\widehat{W}$. This implies that if further units are added
  to this distribution, then their marginal covering ratio ceiling
  will always be $0$. This calculation also shows that the covering
  ratio ceiling is $$\frac{(5m+1)(2n+1)}{3(n+1)(m+1)}<\frac{10}{3}.$$
	
  Let us calculate the covering ratio of the distribution. It is easy
  to see that each vertex in rows $0,1,4,5,6,9,10,$ $11,\ldots$ is
  reachable, but no vertex in rows $2,3,7,8,\ldots$ is
  reachable. Therefore, the covering ratio is
  $$\frac{(3m+1)(2n+1)}{3(n+1)(m+1)}<2.$$
	
  Now we add units of size $2$ near both ends of the row containing
  pebbles (the lighter, smaller pebbles on Fig.~\ref{fig:ell3}) one by
  one.  The above argument shows that the marginal covering ratio
  ceiling is $0$ in every step, and the covering ratio ceiling becomes
  $$\frac{(5m+1)(2n+1)}{3(n+1)(m+1)+4m}<\frac{(5m+1)(2n+1)}{3(n+1)(m+1)}<\frac{10}{3},$$
  so it is decreasing. On the other hand, it is easy to see that one
  can move $4$ pebbles to any vertex of row $5k$, thus every vertex of
  the grid becomes reachable. Therefore the covering ratio is also
  $$\frac{(5m+1)(2n+1)}{3(n+1)(m+1)+4m}$$ which is close to
  $\frac{10}{3}$ if $n$ and $m$ are large enough. So while the
  covering ratio ceiling is decreasing, the covering ratio is
  increasing from $2$ to $\frac{10}{3}$.
\end{proof}	
	
\section{Optimal fractional covering ratio of integer distributions}

The above arguments lead to an interesting question. What is the best
upper bound we can hope for using \cite[Lemma 14]{XY}? The idea behind
this lemma is that the fractional covering ratio is an upper bound on
the covering ratio, and if the starting pebble distribution has only
integer number of pebbles, then there must be a certain amount of excess weight. It is
easy to prove that the optimal fractional covering ratio on the grid
is $9$. The optimal distribution is obtained by placing $1/9$ pebbles
at every vertex. How does this change if we only consider integer
distributions? Let us call this the optimal \emph{integer fractional
covering ratio}, and denote it by $\ifcov(G)$. We give upper and
lower bounds for this ratio in case of the $n\times n$ grids 
.

\begin{theorem}\label{tetel} For any $\varepsilon>0$ there exists
  $n(\varepsilon)$ such that if $n>n(\varepsilon)$ then
\[ 7-\varepsilon \le {\ifcov}(G_{n\times n}) \le \frac{213}{25}=8.52.
\]
\end{theorem} 

\begin{proof} Consider the distribution given in
  Fig.~\ref{fig:frac}. It is easy to calculate, that the number of
  pebbles is $n^2/7+O(n)$, therefore if $n$ is large enough, then
  $\ifcov(G_{n\times n}) \ge 7-\varepsilon$. We need to show, that
  every vertex of the grid is covered by this distribution in the
  fractional sense, i.e. the weight of each vertex is at least
  one. This is clearly true for the vertices, where a pebble is
  placed. By the structure of the distribution, it is clear that it is
  enough to show this for the vertices marked with $A,B,C$. For $A$:
  there is $1$ pebble at distance $1$, $1$ at distance $2$, $1$ at
  distance $3$ and $3$ at distance $4$. Thus the weight at $A$ is at least $
  \frac12+\frac14+\frac18+\frac3{16}>1$. Similar calculations show
  that the weight at $B$ is at least $
  \frac24+\frac38+\frac1{16}+\frac{2}{32}=1$, and the weight at $C$ is
  at least $ \frac{1}{2}+\frac14+\frac18+\frac3{16}>1$. For vertices near
  the border, the extra pebbles placed on the border guarantees the
  required weight. Hence the lower bound is proven.

To prove the upper bound, the following Lemma is needed.
\begin{lemma}\label{lem2} If a distribution covers all vertices of the
  grid in the fractional sense and the vertex $v$ contains a unit of
  size $k$ then the excess weight at this vertex is at least
  $\frac{12}{25}k$.
\end{lemma}

\begin{proof}
  If $k>1$ then the excess weight is at least $k-1>\frac{12}{25}k$, so
  the claim clearly holds. Therefore we only have to deal with the
  $k=1$ case.  The contribution of this pebble to the weight of its
  neighbours is $\frac12$, thus the contribution of other pebbles must
  also be at least $\frac12$. Similarly, the contribution of this
  pebble is $\frac14$ to the vertices at distance 2 from $v$, thus the
  contribution of other pebbles must also be at least $\frac34$. So
  consider the distance 2 neighborhood of $v$ and partition all
  vertices of the grid according to Fig.~\ref{fig:frac-also}. For
  simplicity denote by $x_i$ ($y_i$) the total weight contribution of
  all pebbles in region $X_i$ ($Y_i$) to the corresponding
  vertex. Consider now for example vertex $x_1$. The weight from
  pebbles in $X_1$ is $x_1$, the weight from pebbles in $Y_1$ is
  clearly $\frac12 y_1$, since the distance of any pebble in $Y_1$ and
  $x_1$ is one more than the distance to $y_1$. Similarly, the weight
  from $X_2$ to $x_1$ is $\frac14x_2$, from $Y_2$ to $x_1$ is
  $\frac18y_2$, etc. Since the total weight at $x_1$ must be at least
  $1$ and the contribution of $v$ is $\frac12$,
\begin{equation*}
1\le \frac12+x_1+\frac14x_2+\frac14x_3+\frac14x_4+\frac12y_1+\frac18y_2+\frac18y_3+\frac12y_4
\end{equation*}
must hold.
For all $x_i$ and $y_i$ we can obtain similar inequalities:
\begin{align*}
1\le &\frac12+x_1+\frac14x_2+\frac14x_3+\frac14x_4+\frac12y_1+\frac18y_2+\frac18y_3+\frac12y_4\\
1\le &\frac12+\frac14x_1+x_2+\frac14x_3+\frac14x_4+\frac12y_1+\frac12y_2+\frac18y_3+\frac18y_4\\
1\le &\frac12+\frac14x_1+\frac14x_2+x_3+\frac14x_4+\frac18y_1+\frac12y_2+\frac12y_3+\frac18y_4\\
1\le &\frac12+\frac14x_1+\frac14x_2+\frac14x_3+x_4+\frac18y_1+\frac18y_2+\frac12y_3+\frac12y_4\\
1\le &\frac14+\frac12x_1+\frac12x_2+\frac18x_3+\frac18x_4+y_1+\frac1{4}y_2+\frac1{16}y_3+\frac1{8}y_4\\
1\le &\frac14+\frac18x_1+\frac12x_2+\frac12x_3+\frac18x_4+\frac1{4}y_1+y_2+\frac1{4}y_3+\frac1{16}y_4\\
1\le &\frac14+\frac18x_1+\frac18x_2+\frac12x_3+\frac12x_4+\frac1{16}y_1+\frac1{4}y_2+y_3+\frac1{4}y_4\\
1\le &\frac14+\frac12x_1+\frac18x_2+\frac18x_3+\frac12x_4+\frac1{4}y_1+\frac1{16}y_2+\frac1{4}y_3+y_4
\end{align*}
The vertex $v$ contains $1$ pebble, so the weights coming from other
vertices will give excess weight at $v$. This excess weight is
$\frac12(x_1+x_2+x_3+x_4)+\frac14(y_1+y_2+y_3+y_4)$. To determine the
minimum value of the excess weight, so that the above inequalities are
satisfied, a linear program can be solved. Using duality, one can
easily verify that this minimum is $\frac{12}{25}$. (The minimum is
taken when $x_i=0$ and $y_i=\frac{12}{25}$.)
\end{proof}

\begin{figure}
    \centering
    {\begin{tikzpicture}[scale=.45]

\input{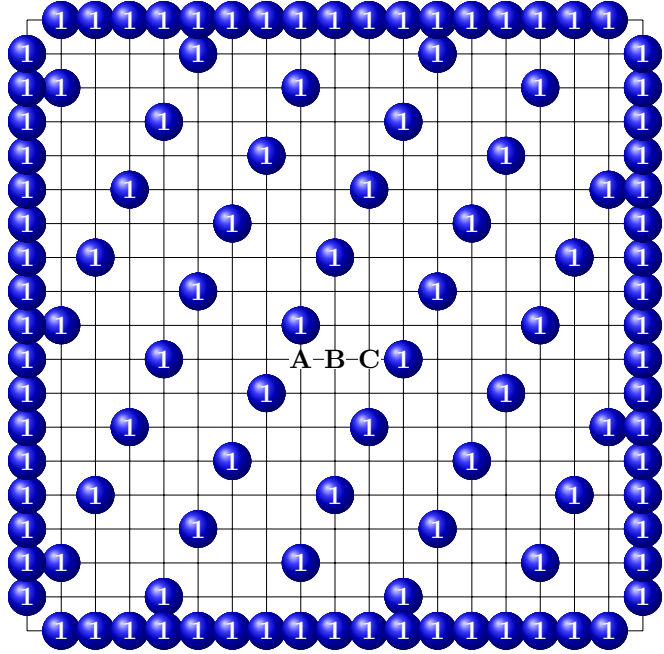}

\foreach \x in {0,...,25} 
\foreach \y in {0,...,14}
\coordinate (X\x_Y\y) at (\x,\y);

%
\draw [very thin, step=1cm] (0,0) grid (18,18);



\foreach \x in {1,...,13} 
	\foreach \y in  {-2,...,4}  {
	\pgfmathtruncatemacro{\i}{(\x+3*\y)}
	\pgfmathtruncatemacro{\j}{(2*\x-\y
	)}
\ifnum \i>0
\ifnum \j<18
\ifnum \i<18
\ifnum \j>0
\node [pebble] (N_\x_\y) at (\i,\j) {$\mathbf{1}$}
\fi\fi\fi\fi;
}
\foreach \x in {1,...,17} {
\node [pebble] (Na_\x) at (\x,0) {$\mathbf{1}$};
\node [pebble] (Na_\x) at (\x,18) {$\mathbf{1}$};
}
\foreach \y in {1,...,17} {
\node [pebble] (Na_\y) at (0,\y) {$\mathbf{1}$};
\node [pebble] (Na_\y) at (18,\y) {$\mathbf{1}$};
}
\node [tipus] (plus) at (8,8) {$\mathbf{A}$};
\node [tipus] (plus) at (9,8) {$\mathbf{B}$};
\node [tipus] (plus) at (10,8) {$\mathbf{C}$};

\foreach \p in {0,...,5} {
\pgfmathtruncatemacro{\q}{(\p+1)}
}
\end{tikzpicture}}
    \caption{Pebble distribution with $7-\varepsilon \le \ifcov(G_{n\times n})$  }
    \label{fig:frac}
  \end{figure}

  Let $D$ be any pebble distribution. An easy calculation shows
  that the sum of the weight contributions of a single pebble to all
  other vertices on the grid is $9$, therefore the sum of weights for
  every vertex on the grid cannot exceed $9|D|$. If $D$ covers every
  vertex on the grid then the weight at every vertex is at least $1$,
  however, if the vertex contains $k$ pebbles then the excess weight
  on this vertex is at least $\frac{12}{25}k$ by
  Lemma~\ref{lem2}. Hence $9|D|\ge n^2+\frac{12}{25}|D|$ holds. This
  implies that $\frac{n^2}{|D|}\le
  9-\frac{12}{25}=\frac{213}{25}=8.52$, proving our claim.
\end{proof}

\begin{figure}
    \centering
    {\begin{tikzpicture}[scale=.5]

\input{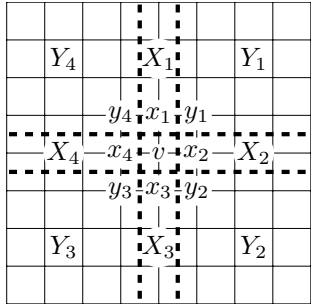}

 \clip (5,5) rectangle (13,13);
\tikzset{koord/.style={draw}};
 \foreach \x in {0,...,18} \node  (\x) at (-1,\x) {\x};
 \foreach \x in {0,...,18} \node (\x) at (\x,-1) {\x};
\foreach \x in {0,...,25} 
\foreach \y in {0,...,14}
\coordinate (X\x_Y\y) at (\x,\y);

\draw [very thin, step=1cm] (0,0) grid (18,18);



\node [csucs] (v) at (9,9) {${v}$};
\node [csucs] (x1) at (9,10) {${x_1}$};
\node [csucs] (x2) at (10,9) {${x_2}$};
\node [csucs] (x3) at (9,8) {${x_3}$};
\node [csucs] (x4) at (8,9) {${x_4}$};
\node [csucs] (y1) at (10,10) {${y_1}$};
\node [csucs] (y2) at (10,8) {${y_2}$};
\node [csucs] (y3) at (8,8) {${y_3}$};
\node [csucs] (y4) at (8,10) {${y_4}$};
\node [csucs] (X1) at (9,11.5) {${X_1}$};
\node [csucs] (X2) at (11.5,9) {${X_2}$};
\node [csucs] (X3) at (9,6.5) {${X_3}$};
\node [csucs] (X4) at (6.5,9) {${X_4}$};
\node [csucs] (Y1) at (11.5,11.5) {${Y_1}$};
\node [csucs] (Y2) at (11.5,6.5) {${Y_2}$};
\node [csucs] (Y3) at (6.5,6.5) {${Y_3}$};
\node [csucs] (Y4) at (6.5,11.5) {${Y_4}$};

\foreach \p in {0,...,5} {
\pgfmathtruncatemacro{\q}{(\p+1)}
}
\draw [szaggatott](8.5,18) -- (8.5,0);
\draw [szaggatott](9.5,18) -- (9.5,0);
\draw [szaggatott](0,9.5) -- (18,9.5);
\draw [szaggatott](0,8.5) -- (18,8.5);
\end{tikzpicture}}
    \caption{Proof of the upper bound}
    \label{fig:frac-also}
  \end{figure}

  Theorem~\ref{tetel} implies that the best upper bound for the
  (integer) covering ratio we can hope using 
the approach of integer fractional covering is $7$.

\section*{Acknowledgement}

The first and second authors are partially supported by the National Research, Development and Innovation Office
NKFIH (grant number K116769).
The second and third authors are partially supported by the National Research, Development and Innovation Office
NKFIH (grant number K108947).

\end{document}